\documentclass[a4paper,fleqn,10pt]{article}
\usepackage{a4wide}
\usepackage{graphicx}

\usepackage{hyperref}

\usepackage[british]{babel}

\usepackage{amsmath,amsfonts,amssymb}
\usepackage{mathabx}
\usepackage{enumitem}
\usepackage{tikz}
\usetikzlibrary{patterns}
\usetikzlibrary{intersections}
\usetikzlibrary{calc}

\usepackage[mathscr]{euscript}

\newcommand{\N}{\mathbf{N}}

\newcommand{\R}{\mathbf{R}}
\newcommand{\RP}{\mathbf{RP}}
\newcommand{\C}{\mathbf{C}}
\DeclareFontFamily{OT1}{pzc}{}
\DeclareFontShape{OT1}{pzc}{m}{it}{<-> s * [1.200] pzcmi7t}{}
\DeclareMathAlphabet{\mathpzc}{OT1}{pzc}{m}{it}
\renewcommand{\P}{\mathbf{PGL}(d)} 
\newcommand{\A}{\mathpzc{A}} 
\renewcommand{\a}{\mathpzc{a}} 
\newcommand{\G}{\mathpzc{G}} 
\newcommand{\g}{\mathpzc{g}} 
\newcommand{\Fr}{\mathpzc{P}} 
\newcommand{\fr}{\mathpzc{p}} 
\newcommand{\B}{\mathpzc{B}} 
\renewcommand{\b}{\mathpzc{b}} 
\newcommand{\T}{\mathpzc{T}} 
\renewcommand{\t}{\mathpzc{t}} 
\newcommand{\F}{\mathpzc{F}} 
\newcommand{\f}{\mathpzc{f}} 
\newcommand{\SP}{\mathpzc{S}} 
\renewcommand{\sp}{\mathpzc{s}} 
\newcommand{\FRK}{\mathpzc{R}} 
\newcommand{\frk}{\mathpzc{r}} 
\newcommand{\SN}{\mathpzc{N}} 
\newcommand{\sn}{\mathpzc{n}} 
\providecommand{\rk}{\mathop{\rm rk}\nolimits} 
\newcommand{\GL}{\mathbf{GL}(d+1)} 
\renewcommand{\O}{\mathbf{O}} 
\newcommand{\Cl}{\mathbf{Cl}} 
\newcommand{\Bl}{\mathbf{Bl}} 
\newcommand{\Gr}{\mathbf{Gr}}
\newcommand{\Id}{\mathbf{I}} 

\usepackage[amsmath,thref,thmmarks,hyperref]{ntheorem}
\newtheorem{theorem}{Theorem}[section] 
\newtheorem{lemma}[theorem]{Lemma}     
\newtheorem{corollary}[theorem]{Corollary}
\newtheorem{proposition}[theorem]{Proposition}

\theorembodyfont{\upshape}
\newtheorem{example}[theorem]{Example}
\theoremheaderfont{\scshape}
\theoremstyle{nonumberplain}
\theoremsymbol{\raisebox{1pt}{\scriptsize\ensuremath{\Box}}}
\newtheorem{proof}{Proof}

\title{Manifolds of Projective Shapes}
\author{T. Hotz, F. Kelma and J. T. Kent}

\begin{document}
\maketitle

\begin{abstract}
The projective shape of a configuration of $k$ points or ``landmarks'' in $\RP^d$ consists of the information that is invariant under projective transformations and hence is reconstructable from uncalibrated camera views. 
Mathematically, the space of projective shapes for these $k$ landmarks can be described as the quotient space of $k$ copies of $\RP^d$ modulo the action of the projective linear group $\P$.
Using homogeneous coordinates, such configurations can be described as real $k\times(d+1)$-dimensional matrices given up to left-multiplication of non-singular diagonal matrices, while the group $\P$ acts as $\GL$ from the right.
The main purpose of this paper is to give a detailed examination of the topology of projective shape space, and, using matrix notation, it is shown how to derive subsets that are in a certain sense maximal, differentiable Hausdorff manifolds which can be provided with a Riemannian metric. 
A special subclass of the projective shapes consists of the Tyler regular shapes, for which geometrically motivated pre-shapes can be defined, thus allowing for the construction of a natural Riemannian metric.
\end{abstract}

\section{Introduction}

The space of projective shapes $\a_d^k$ of $k$ landmarks in $d$-dimensional real projective space $\RP^d$ is of interest in computer vision.
It is commonly defined as the topological quotient of the product of $k$ copies of $\RP^d$ modulo the landmark-wise action of the projective linear group $\P$.
This space arises naturally in the single view uncalibrated pinhole camera model: when taking a $d$-dimensional picture in $\R^{d+1}$ of a $d$-dimensional object without knowledge of any camera parameters such as focal length, angle between the object hyperplane and film hyperplane, etc., then the original object can only be reconstructed up to a projective transformation.
Similarly, it arises in the multiple view uncalibrated pinhole camera model: when taking multiple uncalibrated $d$-dimensional pictures of an object in $\R^{d+1},$ the original configuration of landmarks can only be reconstructed up to a projective transformation.
For details, we refer the reader to the literature, e.g.\ \cite{FL,HZ}.

Other spaces of interest in computer vision include similarity and affine shape spaces.
In shape spaces, one would often like to make metric comparisons, which requires e.g.\ the structure of a Riemannian manifold.
For affine or similarity shapes, the topology of the shape space is well understood and there are natural choices for a Riemannian metric.
Similarity shape space is a CW complex after removing the trivial shape \cite{KBCL}, while affine shape space has a naturally ordered stratification with each stratum being diffeomorphic to a Grassmannian \cite{GT,PM}.
In both cases, the topological subspace of shapes with trivial isotropy group, i.e.\ the shape space of the configurations on which the group action is free, has a natural structure of a Riemannian manifold.

In the case of projective shapes, it turns out that the topological subspace of shapes with trivial isotropy group cannot be given the structure of a Riemannian manifold since it is only a differentiable T1 manifold, but not Hausdorff.
Hence, we have to look for other topological subspaces, which can be endowed with a Riemannian metric. This search is the main purpose of this article.

Besides the quest for a Riemannian structure, there are more desirable properties for a ``good'' topological subspace: 
\begin{enumerate}[label=(\alph*)]
	\item it should be a differentiable Hausdorff manifold with complete Riemannian metric;
	\item it should be closed, and the Riemannian metric invariant, under reordering of the landmarks in the configuration $p=(p_1,\dots,p_k)\in\bigl(\RP^d\bigr)^k$ (relabeling);
	\item when containing a degenerate shape, i.e.\ a shape with non-trivial projective subspace constraints (see Section~\ref{notation}), it should also contain all less degenerate shapes; we will then say that the topological subspace respects the hierarchy of projective subspace constraints;
	\item it should contain as many shapes as possible in the sense that adding further shapes results in the violation of at least one of the properties (maximality).
\end{enumerate}

To our knowledge, there are only two established ways to obtain topological subspaces fulfilling some of these properties, which will be discussed in Section~\ref{approaches}.
Firstly, one can take only those shapes whose first $d+2$ landmarks are in general position and thus form a so-called projective frame.
This topological subspace is homeomorphic to the product of $k-d-2$ copies of $\RP^d$ \cite{MP};
in particular it respects the hierarchy of projective subspace constraints while being maximal, Hausdorff and a differentiable manifold, i.e.\ locally Euclidean with smooth transition maps and second-countable.
Unfortunately, it is not closed under relabeling.
Secondly, one can take all those shapes whose projective subspace constraints fulfill a certain regularity condition, called Tyler (fully-)regular \cite{KM}.
This topological subspace is Hausdorff, closed under relabeling, respects the hierarchy of projective subspace constraints and, as we show in Section~\ref{HausdorffSubsets}, it is a differentiable manifold.
However, these topological subspaces have been constructed in an ad hoc fashion. As of now there is no systematic approach to obtain ``good'' topological subspaces based on the geometrical and topological properties of projective shape space.

In this paper, we therefore analyze the topology of projective shape space in detail.
After recalling some basic facts, fixing our notation in Section~\ref{notation}, discussing the simplest non-trivial case in Section~\ref{case14} and prior approaches in Section~\ref{approaches}, we show which shapes can be separated from each other in the T1 sense, i.e., at least one lies outside some open neighborhood of the other, in Section~\ref{free}.
In particular, we will show that the topological subspace of free shapes, i.e.\ those with trivial isotropy group, is T1 and a differentiable manifold.
Since there are free shapes without a projective frame, frames do not suffice to construct charts on this subspace.
We thus generalize the notion of a frame to obtain charts.
In Section~\ref{HausdorffSubsets}, we show that two shapes which cannot be separated in the Hausdorff sense are already degenerate in a particular way.
This allows us to characterize a reasonable family of differentiable Hausdorff manifolds in Section~\ref{SubNum} which additionally possess properties (b), (c), and (d).
In Section~\ref{Geometry}, we give a geometric justification for Tyler standardization of Tyler regular shapes  \cite{KM} and a Riemannian metric on this topological subspace. 

\section{Preliminaries and notation}\label{notation}

For $d>0,$ real projective space $\RP^d$ is defined as the topological quotient of $\R^{d+1}\setminus\{0\}$ modulo the multiplicative group $\R^*=\R\setminus\{0\},$ i.e.\
\begin{equation}
\RP^d=\bigl\{p=\{\lambda P : \lambda\in\R^*\} : P\in\R^{d+1}\setminus\{0\}\bigr\},
\end{equation}
so it can be seen as the space of lines through the origin in $\R^{d+1}$.
If $p\in\RP^d$ is represented by $P\in\R^{d+1}$, $p=\{\lambda P : \lambda\in\R^*\},$ one calls $P$ homogeneous coordinates for $p$.
A projective subspace of $\RP^d$ of dimension $n<d$ is then the set of lines lying in an $(n+1)$-dimensional linear subspace of $\R^{d+1}$.
Analogously, one can define the projective span of points in $\RP^d$ as the set of lines lying in the linear span of some representatives of the points in $\R^{d+1}$.
These objects are studied in projective geometry, see e.g.\ \cite{RG}.

As the action of the general linear group $\GL$ on $\R^{d+1}$ commutes with the action of $\R^*$, there is a well-defined action of $\GL$ on $\RP^d$ by letting it act on homogeneous coordinates in $\R^{d+1}$.
Since the action of a matrix on $\RP^d$ does not change when multiplying the matrix by a non-zero scalar, the action of $\GL$ is identical with the action of the projective linear group $\P = \GL \mathbin\big/ \R^*$.
This action is naturally carried forward to the product space of configurations with $k$ landmarks
\begin{equation}
\A_d^k = \big( \RP^d \big)^k = \RP^d\times\dots\times\RP^d
\end{equation}
by letting it act component-wise.
Note that projective transformations, i.e.\ the elements of $\P$, map projective subspaces of $\RP^d$ to projective subspaces of the same dimension, i.e.\ points to points, lines to lines etc.
So, if $p\in\A_d^k$ is a configuration having three landmarks on a line, then the images of these three landmarks under a projective transformation also lie on a line.

For $d\geq 1$ and $k\geq d+3$, the space of projective shapes of $k$ landmarks in $\RP^d$ is defined to be the quotient space
\begin{equation}
\a_d^k = \big( \RP^d \big)^k\mathbin\big/\P
\end{equation}
together with the quotient topology.
Since the projection map $\pi:\A_d^k\rightarrow\a_d^k$ is open, the topology of $\a_d^k$ is also second countable; it thus can be characterized by sequences, just like $\A_d^k$.
Further, we can represent a configuration $p\in\A_d^k$ in homogeneous coordinates: up to left-multiplication with a diagonal $k\times k$-dimensional matrix with non-zero real entries, the $k$ landmarks in $\RP^d$ can be represented as a real $k\times(d+1)$-dimensional matrix $P$ whose non-trivial rows $P_{i\cdot}\in\R^{d+1}$, $i=1,\dots,k,$ represent the landmarks in $\RP^d$.
The corresponding equivalence class $[P]$, i.e.\ the shape of $P$, consists of all matrices of the form $DPB$ with $D$ being a non-singular diagonal $k\times k$-dimensional matrix, $B$ a non-singular $(d+1)\times (d+1)$-dimensional matrix, i.e.,
\begin{equation}
[P] = \bigl\{ DPB : D\in\mathbf{GL}(k) \mbox{ diagonal}, B\in\GL \bigr\}.
\end{equation}
Throughout this article, we denote a configuration $p \in (\RP^d)^k$ by a lower case letter, its matrix representation $P \in \R^{k \times (d+1)}$ by the corresponding upper case letter and the shape of $p$ resp.\ $P$ by $[p]$ resp.\ $[P]$.
In abuse of language, we will call $P$ a configuration, too.
Further, we define the rank $\rk p$ of a configuration $p$ to be the rank of any corresponding matrix $P$.
Note that the rank is invariant under $\P$ and thus also well-defined on $\a_d^k$.

Our aim is to find topological subspaces of $\a_d^k$ that can be given the structure of a Riemannian manifold. Topologically speaking, these topological subspaces need to be differentiable Hausdorff manifolds,
as those can be given the structure of a Riemannian manifold \cite{Lee}.

Unfortunately, the space of all projective shapes $\a_d^k$ is not a differentiable Hausdorff manifold, and indeed it is not even T1.
This is easily seen by considering the open neighborhoods of the trivial shape where all landmarks coincide.
Any open neighborhood of the trivial shape is actually already the full space $\a_d^k$.
This phenomenon occurs in similarity and affine shape space as well \cite{KBCL, GT, PM}.

As an example of such a non-Hausdorff space consider similarity shapes of $k$ landmarks in $\R^2$, i.e.\ configurations in $\R^2$ up to similarity transformations.
The corresponding shape space is the topological quotient
\begin{equation*}
\mathbf{sim}_2^k = \bigl( \R^2 \bigr)^k \mathbin\Big/ \bigl( \mathbf{SO}(2)\times \R^* \ltimes \R^2 \bigr)
\label{eq:Simk2}
\end{equation*}
with $\R^2$ acting as translations, $\R^*$ acting as rescaling, $\mathbf{SO}(2)$ acting as rotation of configurations in $\bigl(\R^2\bigr)$, and $\ltimes$ denoting the semi-direct product.
Due to the presence of translations we may consider only centralized configurations---or equivalently configurations of $k-1$ landmarks---up to rotations and rescaling; 
by identifying $\R^2$ with $\C$ we conclude
\begin{equation*}
\mathbf{sim}_2^k \cong \bigl( \R^2 \bigr)^{k-1} \mathbin\Big/ \bigl( \mathbf{SO}(2)\times \R^* \bigr) \cong \C^{k-1} \mathbin\big/ \C^*.
\label{eq:Simk2}
\end{equation*}
Then by rescaling, any equivalence class of centralized configurations is arbitrarily close to the trivial configuration $p_0=(0,\ldots,0)\in\C^{k-1}$, for which all landmarks coincide with the origin.
Hence, $\mathbf{sim}_2^k$ is not Hausdorff.
However, if the trivial shape $[p_0]$ is omitted the resulting similarity shape space $\mathbf{sim}_2^k \setminus \bigl\{[p_0]\bigr\}$ is Hausdorff since $\mathbf{sim}_2^k \setminus \bigl\{[p_0]\bigr\}$ is a complex projective space:
\begin{equation*}
\mathbf{sim}_2^k \setminus \bigl\{[p_0]\bigr\} \cong  \C^{k-1} \setminus\{0\} \mathbin\big/ \C^* = \mathbf{CP}^{k-2}.
\label{eq:Simk2}
\end{equation*}
Thus, it is natural to omit the trivial shape from $\mathbf{sim}_2^k$.

Before we turn to analyze projective shape space $\a_d^k$ in detail, we introduce some interesting topological subspaces of $\A_d^k$ (resp.\ $\a_d^k$): \begin{itemize}\itemsep.3em

\item[$\G_d^k$,] which contains a configuration $p = (p_1, \dots, p_k) \in \A_d^k$ if and only if the landmarks $p_1, \dots, p_k \in \RP^d$ are in \emph{\textbf{g}eneral position}, i.e., no $m$-dimensional projective subspace of $\RP^d$ with $0 \leq m < d$ contains more than $m + 1$ of the landmarks, i.e., any $d + 1$ of the landmarks in $p$ span $\RP^d$.
Analogously, a matrix configuration $P$ is in general position if any subset of $d+1$ rows of $P$ is linearly independent.
An element of $\G_d^{d+2}$ is called a \emph{(projective) frame}. Note that $\G_d^k$ is dense in $\A_d^k$.

\item[$\B_d^k$,] which contains a configuration $p\in \A_d^k$ if and only if the first $d+2$ landmarks in $p$ form a frame, i.e., if and only if $(p_1, \dots, p_{d+2}) \in \G_d^{d+2},$ hence $\G_d^k\subset\B_d^k$.
Frames allow us to define the equivalent of \emph{\textbf{B}ookstein coordinates}~\cite{DM} for similarity shapes, see Lemma~\ref{Bookstein} and \cite[p.~1672; $\B_d^k$ being called $G(k,d)$ there]{MP}.
\item[$\Fr_d^k$,] which contains a configuration $p \in \A_d^k$ if and only if it contains at least one frame, i.e., if and only if there exists a \textbf{p}ermutation $\sigma \in S_k$ of the landmarks such that $\sigma (p) \in \B_d^k,$ thus $\B_d^k\subset\Fr_d^k$ \cite[Remark~2.1; $\Fr_d^k$ being called $\mathcal{FC}_d^k$ there]{MP}.
\item[$\F_d^k$,] which contains a configuration $p \in \A_d^k$ if and only if it has trivial isotropy group, i.e., $\{g\in\P: gp=p \}=\{e\}$.
Elements with trivial isotropy group are called \emph{\textbf{f}ree} or \emph{regular.} Note that $\Fr_d^k\subseteq\F_d^k$ as shown by Mardia and Patrangenaru \cite{MP}.
Analogously, a matrix configuration $P$ is called free if the isotropy group of~$P$ is $ \{ (D,B) : DPB=P \} = \bigl\{ (\lambda \Id_k, \lambda^{-1} \Id_{d+1} ) : \lambda\in\R\setminus\{0\} \bigr\}$ with $\Id_i$ denoting the $i$-dimensional unity matrix.
\item[$\FRK_{\,d}^k$,] which contains a configuration $p \in \A_d^k$ if and only if any corresponding matrix configuration $P$ is of full \textbf{r}ank, i.e., there is no projective subspace of dimension $m<d$ which contains all landmarks.
\item[$\SP_d^k$,] which contains a configuration $p \in \A_d^k$ if and only if it is \emph{\textbf{s}plittable}, i.e., there is a subset $I\subsetneq \{ 1,\ldots,k \}$ s.t.\ $ \rk p_I + \rk p_{I^c} \leq d+1$ where $I^c=\{1,\dots,k\}\setminus I$ and $p_I$ denotes the restriction of $p$ to landmarks with index $i\in I$.
Equivalently, a matrix configuration is splittable if and only if it is not of full rank or its set of rows decomposes into two or more subsets spanning linearly independent subspaces.
Note that $\A_d^k\setminus\FRK_{\,d}^k\subset\SP_d^k$ (take $I=\{1\}$).
\item[$\T_d^k$,] which contains a configuration $p \in \A_d^k$ if and only if any $j$-dimensional projective subspace of $\RP^d$, $j=0,\dots, d-1$ contains fewer than $k \tfrac{j+1}{d+1}$ landmarks.
These configurations are called \emph{\textbf{T}yler \mbox{(fully-)}\-regular} by Kent and Mardia \cite{KM}.
Equivalently, a matrix configuration $P$ is called Tyler regular if and only if any $(j+1)$-dimensional subspace of $\R^{d+1}$ contains fewer than $k \tfrac{j+1}{d+1}$ of the rows of $P$.
\end{itemize}
Note that each of these subspaces is closed under the action of $\P$.
We hence denote the set of equivalence classes by a lower case letter, the corresponding set of configurations by an upper case letter, for example $\A_d^k, \B_d^k$ etc.\ for the configuration spaces, $\a_d^k, \b_d^k$ etc.\ for the corresponding shape spaces.

We say that a configuration $p\in\A_d^k$ \emph{fulfills the projective subspace constraint} $(I,j)$ for a subset $I\subseteq \{1,\ldots,k\}$ of size $|I|\geq j$, $1\leq j<d+1$, if and only if there is a projective subspace $S\subset\RP^d$ of dimension $j-1$ such that $p_i \in S$ for all $i\in I$, i.e.\ $\rk p_I\leq j$, or equivalently, if and only if for any corresponding matrix configuration $P$ there is a $j$-dimensional linear subspace $V$ of $\R^{d+1}$ such that the rows $P_{i\cdot}$ of $P$ are elements of $V$ for all $i\in I$.
We denote the \emph{collection of projective subspace constraints} fulfilled by a configuration $p\in\A_d^k$ by $C(p)=\big\{(I,j) : p \mbox{ fulfills } (I,j) \big\}$.
We call a projective subspace constraint $(I,j)\in C(p)$ \emph{trivial} if $I\subseteq \{1,\ldots,k\}$ is a subset of size $|I| = j$, and \emph{non-trivial} otherwise.
Further, we call $(I,j)\in C(p)$ \emph{splittable} in $C(p)$ if there are $(I_1,j_1),(I_2,j_2)\in C(p)$ with $j_1+j_2=j$, $I_1\cup I_2=I,$ $I_1\cap I_2=\emptyset$.
Thus a configuration $p$ is splittable, i.e., $p\in\SP_d^k$, if and only if $(d+1,\{1,\dots,k\})$ is splittable (slightly generalizing our notation).
We noted before that $C(p)$ is invariant under $\P$, i.e., $C(p)=C(\alpha p)$ for all $\alpha\in\P$, whence $C(p)$ is a property of the projective shape $[p]$.

\section{The case $d=1$, $k=4$}\label{case14}

To motivate the approach taken in this article, let us start with 
the simplest nontrivial configuration space:  four landmarks in $\RP^1$.
For the sake of argument, let $P\in\A_1^4$ be a configuration with $P_{1\cdot}=(x_1,1),P_{2\cdot}=(x_2,1),P_{3\cdot}=(x_3,1),P_{4\cdot}=(x_4,1)$, say.
It is convenient to distinguish five types of configuration:

\begin{enumerate}[label=(\alph*)]
\item all four points distinct, $|\{ x_1, x_2, x_3, x_4 \}| = 4$, i.e., $C(P)$ contains only trivial subspace constraints;
\item single pair coincidence, e.g. $x_4 \neq x_1 = x_2 \neq x_3$, i.e., $(\{1,2\},1)\in C(P)$ is the only non-trivial subspace constraint;
\item double pair coincidence, e.g. $x_1 = x_2 \neq x_3 = x_4$, i.e., $(\{1,2\},1),$ $(\{3,4\},1)\in C(P)$ are the only non-trivial subspace constraints;
\item triple pair coincidence, e.g. $x_1 = x_2 = x_3 \neq x_4$, i.e., $(\{1,2,3\},1)$, $(\{1,2\},1)$, $(\{1,3\},1)$, $(\{2,3\},1)\in C(P)$ are the only non-trivial subspace constraints;
\item quadrupal pair coincidence, $x_1 = x_2 = x_3 = x_4$, i.e.\ $(\{1,2,3,4\},1)\in C(P)$.
\end{enumerate}
Any two configurations with the same projective shape will have the same type of coincidence, i.e.\ projective subspace constraints, so projective shapes $[P]\in\a_1^4$ can be uniquely assigned to types (a)--(e).

It is well-known, see e.g.\ \cite{KM}, that, at least for shapes in $\g_d^k$ (type (a)), the projective shape of $P$ can be described by the cross ratio,
\begin{equation*} 
\rho = \frac{(x_1 - x_2) (x_3 -x_4)}{(x_1 - x_3)(x_2 - x_4)}, 
\end{equation*}
and that $\rho$ lies in one of the
intervals $(-\infty,0), (0,1), (1,\infty)$, a Hausdorff space with three connected components.  
However, the relationship between projective shape and the cross ratio becomes more delicate when we move to the projective shapes for types (b)--(e).

First consider type (e): when all landmarks coincide, then any representing matrix $P_{1=2=3=4}$ comprises rows identical up to rescaling;
further, there are a diagonal matrix $D\in\mathbf{GL}(4)$ and a non-singular matrix $B\in\mathbf{GL(2)}$ such that
\begin{equation*}
DP_{1=2=3=4}B=\begin{pmatrix}
	1 & 0 \\
	1 & 0 \\
	1 & 0 \\
	1 & 0 
\end{pmatrix},
\label{eq:}
\end{equation*}
i.e.\ there is just one single projective shape $[p_{1=2=3=4}]$ of type (e).
The cross ratio is not defined in this case and each open neighborhood of $[p_{1=2=3=4}]$ includes all of projective shape space, as we will discuss in Example~\ref{ex:blur}.

Type (d) is very similar.
In this case there are four distinct projective shapes, labeled e.g. $[p_{1=2=3}]$, etc., depending on which landmark is distinct from the rest.
Again, the cross ratio is not defined.
Each open neighborhood of each of the four projective shapes includes all the  projective shapes for cases (a)--(c), see Example~\ref{ex:blur}.

Types (b) and (c) are more interesting.
Type (c) contains three projective shapes, labeled e.g. $[p_{1=2,3=4}]$, depending on how the landmarks are paired.
Type (b) contains six projective shapes, labeled e.g. $[p_{1=2}]$, again depending on how the landmarks are paired.
In each case the cross ratio is well-defined if one allows values in $\R\cup\{\infty\}$, and it takes the value 0, 1 or $\infty$.
However, the same cross ratio now corresponds to several projective shapes.
For example, $\rho=0$ for the three projective shapes $[p_{1=2,3=4}],\ [p_{1=2}],\ [p_{3=4}]$.
Further, any open set of $[p_{1=2,3=4}]$ always includes the projective shapes $\ [p_{1=2}],\ [p_{3=4}]$, again see Example~\ref{ex:blur}.
Hence, the resulting topological space cannot be Hausdorff in violation of requirement (a) if all three projective shapes are included.
Even worse, no neighborhoods of $\ [p_{1=2}]$ and $[p_{3=4}]$ are disjoint as we will see in Example~\ref{ex:Hausdorff}, so only one of these shapes can be included if the shapes in general position (type (a)) are to be included.

In summary, when looking for a maximal Hausdorff subset of projective shape space, it is essential to keep all the projective shapes of type (a) due to our requirement (c) and to exclude all the projective shapes of types (d)--(e).
However, it is not possible to keep all of the projective shapes for cases (b)--(c).
Keeping those of type (b) would violate the Hausdorff property and thus requirement (a).
Keeping those of type (c) and thus excluding those of type (b) would violate requirement (c), though.
So, only shapes of type (a) may be kept, resulting in an unconnected Hausdorff differentiable manifold.

As we guide the reader through the possible choices in the general setting, the shape space $\a_1^4$ will always be discussed as our motivating example.

\section{Previous approaches}\label{approaches}

The first statistical approach to analyzing projective shapes \cite{MP} used projective frames which are a well-known concept in projective geometry.
As mentioned before, a frame is an ordered set of $d+2$ landmarks in general position.
The group action of $\P$ is both transitive and free on the space $\G_d^{d+2}$ of frames, i.e., for any two frames there is a unique projective transformation mapping one frame to the other \cite{MP}.
This quickly leads to the following result by mapping the frame in the first $d+2$ landmarks to a fixed frame:

\begin{lemma}[\cite{MP}]\label{Bookstein}
$\b_d^k$ is homeomorphic to $\bigl(\RP^d\bigr)^{k-d-2}$.
\end{lemma}

So, $\b_d^k$ is a differentiable Hausdorff manifold and respects the hierarchy of projective subspace constraints, but is not closed under relabeling.
The closure of $\b_d^k$ under permutations is---by definition---the topological subspace $\fr_d^k$ of shapes with a frame. $\fr_d^k$ is a differentiable manifold, but not Hausdorff for any $d\geq 1$, $k\geq d+3$ as we will see in Proposition~\ref{Hausdorff}.

\begin{corollary}\label{frameMnf}
$\fr_d^k$ is homeomorphic to a $d(k-d-2)$-dimensional differentiable T1 manifold.
\end{corollary}

\begin{proof}
Lemma~\ref{Bookstein} gives homeomorphisms from the topological subspaces of shapes with a frame in a fixed subset of $d+2$ landmarks to $(\RP^d)^{k-d-2}$.
Further, note that these topological subspaces of shapes with a frame in a fixed subset of $d+2$ landmarks are open in $\a_d^k$ and $\fr_d^k$.
Hence, these homeomorphisms are ``manifold-valued'' charts on $\fr_d^k$. 
Ordinary charts on $\fr_d^k$ can easily be obtained by composition with charts on the manifold $(\RP^d)^{k-d-2}$, e.g.\ inhomogeneous coordinates.
These charts are compatible since the transition maps are just multiplications with non-singular matrices as well as division by non-vanishing parameters depending smoothly on the representation matrix.
\end{proof}

Alternatively, one can consider the space of shapes in general position $\g_d^k\subset\b_d^k$ which is also a differentiable Hausdorff manifold, respects the hierarchy of projective subspace constraints, and is closed under relabeling.
The drawback of $\g_d^k$ is that it is not maximal for any $d\geq 1$, $k>4$ as we will see in Section~\ref{SubNum}.

\begin{example}
For $d=1$, a frame consists of three distinct landmarks. Hence, $\B_1^4$ consists of all configurations with distinct first three landmarks and arbitrary fourth landmark. $\b_1^4$ is then homeomorphic to the real projective line $\RP^1$ or---equivalently---the circle.
Meanwhile, $\Fr_1^4$ consists of all configurations with at least three of its landmarks distinct and thus forming a frame (types (a), (b) from Section~\ref{case14}).
$\fr_1^4$ is homeomorphic to a circle with three double points corresponding to the single pair coincidences, which cannot be separated in the Hausdorff sense \cite{KM}.
Finally, $\G_1^4$ consists of configurations with no landmark coincidences (only type (a)), hence $\g_1^4$ is homeomorphic to the circle with three points removed.
\end{example}

A different approach was developed by Kent and Mardia \cite{KM}.
The space $\T_d^k$ of Tyler regular configurations comprises configurations $p$ all of whose projective subspace constraints $(I,j)\in C(p)$ satisfy the inequality $|I|<\tfrac{jk}{d+1}$.
It was shown there that any Tyler regular configuration $p\in\T_d^k$ has a matrix representation $P$ fulfilling
 \begin{equation*}  
P_{i\cdot} P_{i\cdot}^t = \tfrac{d+1}{k} \quad \mbox{ for all } i\in\left\{1,\dots,k\right\}
\end{equation*}
and
\begin{equation*}  
P^t P = \Id_{d+1} 
\end{equation*}
with $\Id_{d+1}$ denoting the $(d+1)\times(d+1)$-dimensional identity matrix.
This so-called Tyler standardization $P$ is unique up to multiplication of the rows $P_{i\cdot}$ by~$\pm 1$ and right-multiplica\-tion by an orthogonal matrix, i.e.\ unique up to a compact group action, and can be viewed as a projective pre-shape.
By considering $PP^t\in\R^{k\times k}$, one can even remove the ambiguity of the $\O(d+1)$-action.
This gives a covering space of the space $\t_d^k$ of Tyler regular shapes.
The covering space is Hausdorff, whence $\t_d^k$ is Hausdorff.

We show in Section~\ref{Geometry}, that $\t_d^k$ is a differentiable Hausdorff manifold;
it is obviously closed under relabeling and respects the hierarchy of projective subspace constraints.
Additionally, we show that $\t_d^k$ is maximal for some, but not all $k$ and $d$.
Note that the approach via frames differs from the approach via Tyler regularity since, for $d\geq3$, there are Tyler regular shapes without a frame, see Figure~\ref{FreeNoFrame}. 

\begin{example}
In the case $d=1$ and $k=4$, $\t_1^4$ consists of shapes with projective subspace constraints $(I,j)$ with $j=1$ and $|I|=1<2=\tfrac{1\cdot 4}{2}$, i.e.\ the shapes in general position (only type (a) from Section~\ref{case14}).
Hence, $\t_1^4=\g_1^4$ with $\g_1^4$ being homeomorphic to the circle with three points removed, as we have seen before.
\end{example}

Neither of these approaches discusses the topological background of these choices. The goal of this article is to shed light on the topology of these topological subspaces of projective shapes.

\section{The manifold of the free}\label{free}

To understand the topology of a topological space $M$, it is vital to know which elements of $M$ cannot be separated from each other by open neighborhoods.
It is common to use the well-known separation axioms to described the degree of separation. Two of those will be discussed here.

A topological space $M$ is said to be \begin{description}[leftmargin=17pt]
\item[$\mathbf{T1}$] if for any two points $p,q\in M$ there are open neighborhoods $U_p$ and $U_q$ of $p$ and $q$ respectively not containing the other point, i.e., $q\notin U_p$ and $p\notin U_q$.
\item[\textbf{Hausdorff} or $\mathbf{T2}$] if for any two points $p,q\in M$ there are disjoint open neighborhoods of $p$ and $q$.
\end{description} 

The intersection of all open neighborhoods to a point $p\in M$ is a useful tool towards understanding the separation properties of a space $M$.
This set was introduced as the \emph{blur} $\Bl(p)$ of $p$ in $M$ by Groisser and Tagare in their discussion of affine shape space \cite{GT}.
We will call a point $p\in M$ \emph{unblurry} if $\Bl(p)=\{p\},$ and \emph{blurry} in the case that its blur is a strict superset of $\{p\}$.

Equivalently, the blur could also be defined via sequences.

\begin{lemma}\label{BlurSeq}
Let $M$ be a topological space and $p,q\in M$. Then, $p\in\Bl(q)$ if and only if the constant sequence $(p)_{n\in\N}$ converges to $q$.
\end{lemma}

\begin{proof}
$p\in\Bl(q)$ if and only if $p$ is in every neighborhood of $q$ which happens if and only if the sequence $(p)_{n\in\N}$ converges to $q$. 
\end{proof}

This concept is closely related to the more familiar concept of closure which has also been pointed out by Groisser and Tagare \cite{GT}.

\begin{lemma}\cite[Lemma 5.2]{GT}
Let $M$ be a topological space and $p,q\in M$. Then, $p\in\Bl(q)$ if and only if $q\in\Cl(p),$ the latter denoting the closure of $\{p\}$ in $M$.
\end{lemma}

In particular, every point is unblurry if and only if every point is closed, which in turn is equivalent to the space being T1 \cite{AP}.
This motivates us to take a closer look at the unblurry shapes. 

As it turns out, a shape is blurry if it is splittable; the converse is also true as we will show after Theorem~\ref{MnfOfFree}.

\begin{proposition}\label{blursplit}
Let $[p]\in\sp_d^k$ be a splittable shape. Then $[p]$ is blurry.
More precisely, there exists a $[q]\in\Bl\bigl([p]\bigr)$ that is less degenerate than $[p]$, i.e.\ $C(q)\subsetneq C(p)$.
\end{proposition}

\begin{proof}
We will use Lemma~\ref{BlurSeq}.
First, consider an arbitrary shape $[P]$ with $ \rk P < d+1 $.
There is a non-singular matrix $B\in\GL$ such that $ PB = (P_1,0_k) $ for some $P_1$ and $0_k$ being a column vector of $k$ zeroes.
Of course, $PB$ is still of shape $[P]$.
Then, the sequence $ \bigl((P_1,z)B_n\bigr)_{n\in\N} $ with $ B_n = \mbox{diag}\bigl(1,\ldots,1,\tfrac{1}{n}\bigr) $ and arbitrary $z\in\R^k$ has limit $PB$.
Hence, $\bigl[(P_1,z)\bigr] \in \Bl([P])$ for any $z\in\R^k,$ while there is a $z\in\R^k$ such that $\rk (P_1,z)>\rk P$.
Therefore, $\Bl([P])\neq \{[P]\}$, whence $[P]$ is blurry.

Now, let $[P]\in\sp_d^k$ be of rank $d+1$ with $(I,j),(I^c,d+1-j)\in C(P)$. 
Since $|I|+|I^c|=k>d+2$, w.l.o.g.\ $|I^c|>j$, else $|I|>d+1-j$.
Then there is a suitable permutation $\sigma$ of the rows of $P$ and a suitable non-singular matrix $B\in \GL$ such that the matrix $\hat{P}=\sigma(P) B$ is a block diagonal matrix
\begin{equation}
\hat{P} = \begin{pmatrix}
	P_1 & 0 \\ 0 & P_2
	\label{eq:block}
\end{pmatrix}
\end{equation}  
for some matrices $ P_1 \in \R^{|I|\times j} $ and $ P_2 \in \R^{|I^c|\times (d+1-j)} $. The sequence given by 
\begin{equation*}
\begin{pmatrix}
	n\Id_{|I|} & 0 \\ 0 & \Id_{|I^c|}
\end{pmatrix}
\begin{pmatrix}
	P_1 & 0 \\ Z & P_2
\end{pmatrix}
\begin{pmatrix}
	\tfrac{1}{n}\Id_{j} & 0 \\ 0 & \Id_{d+1-j}
\end{pmatrix}
= \begin{pmatrix}
	P_1 & 0 \\ \tfrac{1}{n}Z & P_2
\end{pmatrix}
\end{equation*} 
has limit $\hat{P}$ for any $ Z \in\R^{|I^c|\times j} $. Hence,
\begin{equation}
\left[\begin{pmatrix}
	P_1 & 0 \\ Z & P_2
\end{pmatrix}\right] \in \Bl\bigl([\hat{P}]\bigr).
\label{eq:blur}
\end{equation} 
Since $j<|I^c|$, there is a $Z\in\R^{|I^c|\times j}$ which breaks a projective subspace constraint of $[\hat{P}]$, whence $\Bl([\hat{P}])\neq \{[\hat{P}]\}$ and consequently $\Bl([P])\neq \{[P]\}$; thus, $[P]$ is blurry.
\end{proof}

\begin{example}\label{ex:blur}
In the case $d=1$ and $k=4$, the topological subspace $\SP_1^4$ of splittable configurations consists of the trivial configurations (all landmarks identical, type (e) from Section~\ref{case14}) and all those comprising only two different landmarks, when either three landmarks coincide (type (d)) or there are two pairs of landmarks coinciding (type (c)).

The blur of the trivial shape $[P_{1=2=3=4}]$ is
\begin{equation*}
\Bl\bigl([P_{1=2=3=4}]\bigr) = \Bl\left(\left[\begin{pmatrix}
	1 & 0 \\
	1 & 0 \\
	1 & 0 \\
	1 & 0 
\end{pmatrix}\right]\right) = \left\{\left[\begin{pmatrix}
	1 & a \\
	1 & b \\
	1 & c \\
	1 & d 
\end{pmatrix}\right] : a,b,c,d\in\R \right\}
\label{eq:}
\end{equation*} which is the full shape space $\a_d^k$.

The shapes comprising only two different landmarks can be represented as block matrices as in Equation~\eqref{eq:block} e.g.\ by mapping the rows in a representing matrix to the standard basis of $\R^2$ by a matrix $B\in\mathbf{GL}(2)$;
thus,
\begin{equation*}
[P_{2=3=4}] = \left[\begin{pmatrix}
	1 & 0 \\
	0 & 1 \\
	0 & 1 \\
	0 & 1 
\end{pmatrix}\right],\quad
[P_{1=2,3=4}] = \left[\begin{pmatrix}
	1 & 0 \\
	1 & 0 \\
	0 & 1 \\
	0 & 1 
\end{pmatrix}\right], \quad \mbox{etc.}
\end{equation*}
The blur of the shape $[P_{2=3=4}]\in\sp_1^4$ comprises then $[P_{2=3=4}]$ and all shapes with fewer projective subspace constraints as Equation~\eqref{eq:blur} in the preceding proof indicates.
Meanwhile, the blur of $[P_{1=2,3=4}]\in\sp_1^4$ comprises $[P_{1=2,3=4}]$ and the single pair coincidences $[P_{1=2}]$ and $[P_{3=4}]$.
\end{example}

Due to Proposition~\ref{blursplit}, we henceforth limit ourselves to the analysis of those configurations (resp.\ shapes) which are not splittable. Those can be characterized algebraically via the group action.

\begin{proposition}\label{SplitFree}
A configuration is free if and only if it is not splittable, i.e.\ $\F_d^k = \A_d^k \setminus \SP_d^k$.
\end{proposition}

\begin{proof}
If $\rk P < d+1,$ then $P$ is obviously splittable, but not free. Hence, we will focus on configurations with $\rk P = d+1$. \\
Now, assume there are projective subspace constraints $(I,j),(I^c,d+1-j)$ such that $ \rk P_I + \rk P_{I^c} = \rk P = d+1 $. Then there is a permutation $\sigma$ of the rows and a matrix $B\in \GL$ such that $\sigma(P) B$ is a block diagonal matrix $\left(\begin{smallmatrix} \hat{P}_I & 0 \\ 0 & \hat{P}_{I^c} \end{smallmatrix}\right)$.
Hence, $\sigma(P) B$ is not free since 
\begin{equation*}
\begin{pmatrix} \hat{P}_I & 0 \\ 0 & \hat{P}_{I^c} \end{pmatrix}  
= \begin{pmatrix} \lambda\Id_{|I|} & 0 \\ 0 & \Id_{|I^c|} \end{pmatrix}
\begin{pmatrix} \hat{P}_I & 0 \\ 0 & \hat{P}_{I^c} \end{pmatrix} 
\begin{pmatrix} \lambda^{-1}\Id_{j} & 0 \\ 0 & \Id_{d+1-j} \end{pmatrix} 
\end{equation*}
for any $\lambda\in\R^*$.
Therefore, $\sigma(P)$ has a non-trivial isotropy group, hence so has $P$.

For the opposite direction, assume $P$ is not free.
Then there exists an invertible diagonal matrix $D$ and some $B\in\GL,\,B\neq \lambda \Id_{d+1}$, $\lambda\in\R\setminus\{0\},$ such that  $DPB=P$.
Hence, the rows of $P$ are eigenvectors of $B^t$ with corresponding eigenvalues $\lambda_1,\ldots,\lambda_k$, say (taking at most $d+1$ distinct eigenvalues).
There are at least two distinct eigenvalues, else $B = \lambda_1 \Id_{d+1}$ contradicting the assumption.
Then, $(I,\rk P_I), (I^c,\rk P_{I^c})\in C(P)$ with $I=\{i\,:\,\lambda_i=\lambda_1\}$, while $ \rk P_I + \rk P_{I^c} = d+1 $, whence $P$ is splittable. 
\end{proof}

From Propositions~\ref{blursplit} and \ref{SplitFree} we conclude that the subspace $\f_d^k$ of the free shapes is the largest subspace which is T1 and respects the hierarchy of subspace constraints.

In the case $d=1$, the splittable shapes are those comprising at most two distinct landmarks as we have seen before. Thus, Proposition~\ref{SplitFree} states that a shape $[p]\in\a_1^k$ is free if and only if it has at least three distinct landmarks.
Three distinct landmarks always form a frame for $d=1$.
Indeed, Mardia and Patrangenaru \cite{MP} have shown for any $d\geq 1$ that shapes which include a frame are free, i.e.\ $\fr_d^k\subseteq\f_d^k$.
However, the other inclusion does not hold for $d\geq 3$: e.g. for $d=3$, take three lines, which are not coplanar, but have a common intersection point, and put two landmarks on each line, and another on the intersection point.
Such a configuration of seven landmarks is free, but does not contain a frame since there are no five landmarks in general position, see Figure~\ref{FreeNoFrame}(a).
The same argument works when removing the landmark on the intersection point.
Analogously, a free shape without a frame can be constructed for any $d>3$. 

\begin{figure}
\begin{tikzpicture}[scale=0.4]
	\node at (-4.5,.6) {(a)};
	\draw (0,1) -- (0,-5);
	\draw (.7,.7) -- (-4.2,-4.2);
	\draw (-0.7,.7) -- (4.2,-4.2);
	\draw (0.7,0) node {\textbf{1}};
	\draw (-0.9,-1.6) node {\textbf{5}};
	\draw (0.7,-2) node {\textbf{6}};
	\draw (2.3,-1.6) node {\textbf{7}};
	\draw (-2.8,-3.5) node {\textbf{2}};
	\draw (0.7,-4) node {\textbf{3}};
	\draw (4.2,-3.5) node {\textbf{4}};
	\draw[fill=black] (0,0) circle (0.1cm);
	\draw[fill=black] (-1.6,-1.6) circle (0.1cm);
	\draw[fill=black] (0,-2) circle (0.1cm);
	\draw[fill=black] (1.6,-1.6) circle (0.1cm);
	\draw[fill=black] (-3.5,-3.5) circle (0.1cm);
	\draw[fill=black] (0,-4) circle (0.1cm);
	\draw[fill=black] (3.5,-3.5) circle (0.1cm);
	\draw (10,-2) node {$P=\begin{pmatrix}
	1 & 0 & 0 & 0 \\
	0 & 1 & 0 & 0 \\
	0 & 0 & 1 & 0 \\
	0 & 0 & 0 & 1 \\
	1 & 1 & 0 & 0 \\
	1 & 0 & 1 & 0 \\
	1 & 0 & 0 & 1
\end{pmatrix}$};

\node at (17,.6) {(b)};
\draw (19,-2) node {\textbf{$G(P)$}};
\draw[fill=black] (23,-.5) node[anchor=south] {\textbf{1}} circle (0.1cm);
\draw[fill=black] (26,-3.5) node[anchor=north] {\textbf{4}} circle (0.1cm);
\draw[fill=black] (23,-3.5) node[anchor=north] {\textbf{3}} circle (0.1cm);
\draw[fill=black] (20,-3.5) node[anchor=north] {\textbf{2}} circle (0.1cm);
\draw (23,-.5) -- node[anchor=west] {\textbf{6}} (23,-3.5);
\draw (23,-.5) -- node[anchor=west] {\textbf{7}} (26,-3.5);
\draw (23,-.5) -- node[anchor=east] {\textbf{5}} (20,-3.5);
	
\end{tikzpicture}
\caption{(a) A free configuration $P$ in $\RP^3$ without a frame. The seven landmarks lie on three non-coplanar lines; landmark~1 is the intersection point. (b) Its graph $G(P)$ corresponding to the first 4 landmarks in general position.}
\label{FreeNoFrame}
\end{figure}
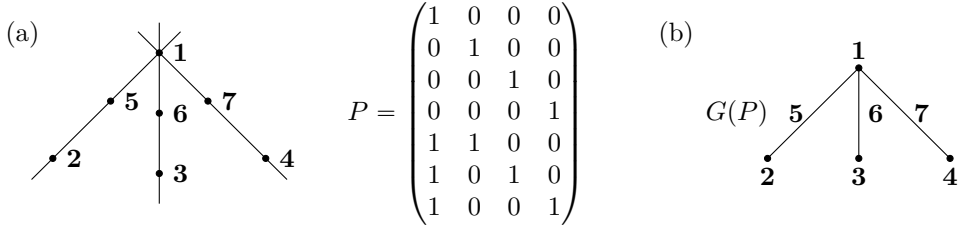

Hence, having a frame is not essential for a shape to be free.
While frames can be used as charts on $\fr_d^k\subseteq\f_d^k$, this is not possible for $\f_d^k$ for $d\geq 3$ since the charts associated with frames do not cover $\f_d^k$ for $d\geq 3$.
However, the notion of a frame can be generalized to obtain charts on $\f_d^k$ as follows.

A free configuration contains at least $d+1$ landmarks in general position since a free configuration is of full rank.
Now, a configuration $P=\big(\begin{smallmatrix}P_0 \\ P_1\end{smallmatrix}\big),$ whose first $d+1$ landmarks $P_0$, say, are in general position, i.e.\ $P_0\in\G_d^{d+1}$, is equivalent to a matrix of the form
\begin{equation}
\tilde{P} = 
\begin{pmatrix}
	\Id_{d+1} \\
	P_*
\end{pmatrix},
\label{graphform}
\end{equation}
where $P_*=P_1 P_0^{-1}$ consists of non-trivial rows.
For such a configuration $P$ define its corresponding (undirected) edge-colored \emph{graph} $G(P) = \big(V(P),E(P)\big)$ by taking the columns of $\tilde{P}$ as vertices, i.e.\ $V(P) = \{ 1,\ldots,d+1 \}$, and letting there be an edge labeled with color ``$l$'' between the vertices $i,j$ if both $\tilde{P}_{li}\neq 0$ and $\tilde{P}_{lj}\neq 0$ (see Figure~\ref{FreeNoFrame}(b) as an example).
Note that there may be multiple edges of different colors between two vertices.
We denote the set of edges of color ``$l$'' by $E_l\subseteq E(P)$, $l\in\{d+2,\ldots,k\}$.

This definition of the graph of a configuration with the first $d+1$ landmarks in general position is well-defined and invariant under $\P$:
let $Q=DPB$ be an equivalent configuration, $D_0$ be the upper left square block of $D$ with $d+1$ rows, $D_1$ be the lower right square block of $D$ with $k-d-1$ rows, $P_0$ be the first $d+1$ landmarks of $P$, $P_1$ be the last $k-d-1$ landmarks of $P$.
Then $Q_*$ in $\tilde{Q}$ is given by 
\begin{equation*}
D_1 P_1 B \bigl(D_0 P_0 B\bigr)^{-1} = D_1 P_1 P_0^{-1} D_0^{-1} = D_1 P_* D_0^{-1}.
\end{equation*} 
Hence, $P_*$ is unique up to left- and right-multiplication by non-singular diagonal matrices.
But these actions do not affect the graph.

This definition can easily be extended to any configuration $P$ with a given set of $d+1$ landmarks in general position.

Now, we can connect freeness with graph properties.

\begin{proposition}\label{Free_GraphConnected}
Let $P$ be a configuration whose first $d+1$ landmarks are in general position. Then $P$ is free if and only if $G(P)$ is connected. 
\end{proposition}

\begin{proof}
If $G(P)$ is not connected, then the columns of $\tilde{P}$ split into two disconnected sets, so $\tilde{P}$ is splittable, as is $P$, hence not free according to Proposition~\ref{SplitFree}.

Now, suppose that $G(P)$ is connected.
We assume that $P=\tilde{P}$ in Equation~\eqref{graphform} w.l.o.g. 
Further, assume that there exist matrices $D=\mbox{diag}(\lambda_1,\ldots,\lambda_k)$ and $B\in\GL$ such that $DPB=P$.
Then, $B = \mbox{diag}(\lambda_1^{-1},\ldots,\lambda_{d+1}^{-1})$ since Equation \eqref{graphform} implies 
\begin{equation*} 
\mbox{diag}(\lambda_1,\ldots,\lambda_{d+1}) \Id_{d+1}B = \Id_{d+1}
\end{equation*} 
for the first $d+1$ rows of $P$.
For any two connected columns $i,j,$ there is a row $P_l$ such that both $P_{li}\neq 0$ and $P_{lj}\neq 0$. 
Then, $P_{li} = \lambda_l P_{li} \lambda_i^{-1}$ and $P_{lj} = \lambda_l P_{lj} \lambda_j^{-1}$.
From this we conclude 
\begin{equation*}
\lambda_l\lambda_i^{-1}=\lambda_l\lambda_j^{-1}=1,
\end{equation*}
and thus $\lambda_1 = \ldots = \lambda_{d+1}$ since all columns are connected, so $D=\lambda_1 \Id_{k}$ and $B=\lambda_1^{-1} \Id_{d+1}$, i.e., $P$ is free.
\end{proof}

In the following, we will call $d+1$ landmarks in general position together with a connected \emph{tree} $G$ with edges labeled with the remaining landmarks a \emph{pseudo-frame}.
So $G$ contains no cycles and gets disconnected if an edge is removed whence it is a minimal substructure of a connected graph.
This generalizes the idea of a ``frame'' since a frame is a pseudo-frame with a connected tree on $d+1$ landmarks in general position where all edges are labeled with the same landmark (see Figure~\ref{Frame}),  i.e., a uni-colored tree gives rise to a frame.
\begin{figure}
\begin{center}
\begin{tikzpicture}[scale=0.4]
	\draw[fill=black] (-2,-2)  node[anchor=north east] {\textbf{4}} circle (0.1cm);
	\draw[fill=black] (2,2) node[anchor=south west] {\textbf{2}} circle (0.1cm);
	\draw[fill=black] (-2,2) node[anchor=south east] {\textbf{3}} circle (0.1cm);
	\draw[fill=black] (2,-2) node[anchor=north west] {\textbf{1}} circle (0.1cm);
	\draw (2,-2) -- node[anchor=north] {\textbf{5}} (-2,-2);
	\draw (2,-2) -- node[anchor=west] {\textbf{5}} (2,2);
	\draw (2,-2) -- node[anchor=east] {\textbf{5}} (0,0) -- (-2,2);
	\draw (-2,-2) -- (0,0) -- node[anchor=east] {\textbf{5}} (2,2);
	\draw (-2,-2) -- node[anchor=east] {\textbf{5}} (-2,2);
	\draw (2,2) -- node[anchor=south] {\textbf{5}} (-2,2);
	\draw (-4.5,0) node {\textbf{$G(P)$}};
	\draw (-12,0) node {$P=\begin{pmatrix}
	1 & 0 & 0 & 0 \\
	0 & 1 & 0 & 0 \\
	0 & 0 & 1 & 0 \\
	0 & 0 & 0 & 1 \\
	1 & 1 & 1 & 1
\end{pmatrix}$};
	
\end{tikzpicture}
\end{center}
\caption{A frame $P$ and its graph $G(P)$ which is a complete graph. All spanning trees of $G(P)$ give a pseudo-frame.}
\label{Frame}
\end{figure}
We will say that a configuration $p$ (resp.\ shape $[p]$) contains a pseudo-frame $\big(\{i_1,\ldots,i_{d+1}\},G\big)$ if $p_{i_1},\ldots,p_{i_{d+1}}$ are in general position and the corresponding graph to this configuration (resp.\ shape) has the tree $G$ as a subgraph.
We conclude from Proposition~\ref{Free_GraphConnected} that every free shape contains a pseudo-frame. 

Since pseudo-frames are a generalization of frames, we obtain a topological Hausdorff subspace when considering all shapes containing a fixed pseudo-frame, thus generalizing the definition of $\b_d^k$ and Lemma~\ref{Bookstein}:
denote the number of edges in the tree $G=(\{i_1,\ldots,i_{d+1}\},E)$ labeled with the landmark $l$ by $|E_l|,$ and define $\# E = \bigl|\{l\,:\,E_l\neq\emptyset\}\bigr|$.

\begin{proposition}\label{pseudobookstein}
The topological subspace of all shapes containing a certain pseudo-frame \newline$\bigl(\{i_1,\ldots,i_{d+1}\},G\bigr)$ is ho\-meo\-mor\-phic to the $d(k-d-2)$-dimensional differentiable Hausdorff manifold 
\begin{equation}\label{eq:pseudobookstein}
\bigl(\RP^d\bigr)^{k-d-1-\# E} \times \bigtimes_{\substack{l=1: \\ l\notin\{i_1,\ldots,i_{d+1}\} \\ E_l\neq\emptyset}}^k \R^{d - |E_l|}.
\end{equation}
\end{proposition}

\begin{proof}
The final factor of the product in Equation~\eqref{eq:pseudobookstein} has dimension $d(\# E-1)$ since $\sum_{l}{|E_l|} = d$ is the number of edges in the tree $G$ with $d+1$ vertices.
This explains the dimension of the manifold.

To show the homeomorphy, consider for a shape $[P]$ (after reordering the rows) a representative of the form in Equation~\eqref{graphform}. Obviously, the rows of $P_*$ which are not used for the graph give us the first factor of the product in Equation~\eqref{eq:pseudobookstein}. By rescaling of rows and columns the non-zero entries determined by the labeled tree are w.l.o.g.\ equal to 1, and the rest of the row may be filled with any real number, hence we obtain $\R^{d+1-(|E_l| + 1)} = \R^{d - |E_l|}$ for row $l$ if $|E_l|\neq 0$.
\end{proof}

Now, Proposition~\ref{pseudobookstein} gives us finitely many, manifold-valued charts for $\f_d^k$ whence it is a differentiable manifold.

\begin{theorem}\label{MnfOfFree}
$\f_d^k$ is a $d(k-d-2)$-dimen\-sional differentiable T1 manifold.
\end{theorem}

\begin{proof}
From Proposition~\ref{pseudobookstein} we obtain homeomorphisms from open subsets of $\f_d^k$ to a differentiable manifold.
When composing those with charts of the differentiable manifold, we obtain charts on $\f_d^k$ whose domains cover the full space.
Since the transition maps between these charts are just multiplications from left and right with non-singular diagonal and non-singular matrices depending smoothly on the representation matrix, the manifold is indeed differentiable.
\end{proof}

We would like to point out that for $d=1$ the concept of pseudo-frames adds no extra insight, since a pseudo-frame is already a frame in this case (any colored tree with $d+1=2$ vertices is uni-colored).
For $d=2,$ any shape with a pseudo-frame already contains a frame, i.e.\ $\f_d^k=\fr_d^k$ for $d=1,2$.
The critical shape to consider in the case $d=2$ is (in the form of Equation~\eqref{graphform})
\begin{equation*}[p]=
\left[\begin{pmatrix}
	1 & 0 & 0 \\
	0 & 1 & 0 \\
	0 & 0 & 1 \\
	u & v & w \\
	x & y & z \\
	& \vdots &
\end{pmatrix}\right].
\end{equation*}
Let w.l.o.g.\ there be a pseudo-frame in the first five rows of $[p]$.
If either all of $u,v,w\neq0$ or all of $x,y,z\neq0$, then $[p]$ contains a frame.
So, let there be a vanishing value in both of the rows.
Since there is a pseudo-frame in the first five rows of $[p]$, there is at most one vanishing value in each row, and it cannot be in the same column.
For the sake of argument, let $w=x=0$.
Then, $p_{\{1,3,4,5\}}$ is a frame.
Thus, charts stemming from frames suffice to cover $\f_2^k,$ while pseudo-frames give a larger atlas on $\f_2^k$.

From Theorem~\ref{MnfOfFree} follows that free shapes $[p]\in\f_d^k$ are unblurry which is the converse direction of Proposition~\ref{blursplit}:
$\f_d^k$ is open in $\a_d^k$ since $\F_d^k$ is open in $\A_d^k$ and $\pi:\A_d^k\rightarrow\a_d^k$ is an open map.
Hence, neighborhoods of $[p]$ in $\f_d^k$ are already neighborhoods of $[p]$ in $\a_d^k$. Now, $\f_d^k$ is T1 by Theorem~\ref{MnfOfFree} whence the intersection of all neighborhoods of $[p]$ in $\f_d^k$ is just $\{[p]\}$, so is the intersection of all neighborhoods of $[p]$ in $\a_d^k$.
Hence, $\Bl([p])=\{[p]\}$.

Unfortunately, the manifold $\f_d^k$ of the free is never Hausdorff for $d\geq 1$ and $k\geq d+3$.
Even the subset $\fr_d^k$ is never Hausdorff for any $d\geq 1$ and $k\geq d+3$ which will follow from Proposition~\ref{Hausdorff}.
Note, however, that all open subsets of $\f_d^k$ are differentiable manifolds by Theorem~\ref{MnfOfFree}.
In particular, all topological subspaces $\mathpzc{y}=\mathpzc{Y}\big/\P$ of $\f_d^k$ respecting the hierarchy of projective subspace constraints are differentiable manifolds, since then any configuration $p\in \mathpzc{Y}$ has an open neighborhood $U\ni p$ with $U\subseteq \mathpzc{Y}$ whence $\mathpzc{Y}$ is open in $\A_d^k$ and $\mathpzc{y}$ is open in $\a_d^k$.

The situation in similarity resp.\ affine shape space is not as complicated \cite{KBCL,PM}: in both cases, the full shape space is not T1.
The largest T1 space in similarity shape space is the full space without the trivial shape, while in affine shape space it is the subspace of the free just like in projective shape space.
In both cases, the subspace of the free is a differentiable manifold and, in contrast to the projective situation, Hausdorff.

Note that Theorem~\ref{MnfOfFree} in connection with the decomposition $\a_d^k=\f_d^k\cup\sp_d^k$ (Proposition~\ref{SplitFree}) leads to a non-Hausdorff stratification of projective shape space since splittable shapes can be understood as products of lower-dimen\-sion\-al free shapes \cite{K3}.

\section{Hausdorff subsets}\label{HausdorffSubsets}

In applications, one is often interested in metric comparisons of different shapes. Therefore, the underlying shape space needs to be a metrizable topological space (e.g.\ a Riemannian manifold) which is---of course---at least Hausdorff. Hence, we are looking for topological Hausdorff subspaces of projective shape space.

Consider a shape $[P]$ which fulfills the projective subspace constraint $(\{1,\dots,i\},j)$, which may be trivial or non-trivial, i.e., $[P]$ has a representative
\begin{equation*}
P = \begin{pmatrix}
	P_1 & 0 \\
	Z & P_2
\end{pmatrix}
\end{equation*}
for some matrices $ P_1 \in \R^{i\times j} ,$ $ P_2 \in \R^{(k-i)\times (d+1-j)} $, and $ Z \in \R^{(k-i)\times j} $ ($Z$ possibly being zero). 
Additionally, consider the sequence $([P_n])_{n\in\N}$ with
\begin{align*}
P_n & = \begin{pmatrix}
	P_1 & \tfrac1n Y \\
	Z & P_2
\end{pmatrix} \\
& = \begin{pmatrix}
	\Id_{i} & 0 \\
	0 & n \Id_{k-i}
\end{pmatrix} \begin{pmatrix}
	P_1 & Y \\
	\tfrac1n Z & P_2
\end{pmatrix} \begin{pmatrix}
	\Id_j &  0 \\
	0 & \tfrac1n \Id_{d+1-j}
\end{pmatrix}
\end{align*} for some $Y\in\R^{i\times (d+1-j)}$.
This sequence converges to $[P]$ and to $[Q]$ with
\begin{equation*}
Q = \begin{pmatrix}
	P_1 & Y \\
	0 & P_2
\end{pmatrix}
\end{equation*}
as $n$ goes to infinity.
But $[P]\neq[Q]$ for some choices for $Y,Z$ as $Y,Z$ may break some projective subspace constraint.
Hence, a topological subspace of $\a_d^k$ containing such $[P],$ $[Q]$ and $[P_n]$ for all $n\in\N$ would not be Hausdorff since sequences in Hausdorff spaces have at most one limit point.
Note that $Q$ fulfills the projective subspace constraint $(\{i+1,\dots,k\},d+1-j)$.

This observation can be strengthened to the following result for determining if a projective subspace of $\a_d^k$ is Hausdorff.

\begin{proposition}\label{Hausdorff}
Let $  \mathpzc{y}\subseteq \frk_d^k$ be a topological subspace which contains $\g_d^k$ and is not Hausdorff.
Then, there are two shapes $[p],[q]\in \mathpzc{y}$ with $[p]\neq[q],$ $(I,j)\in C(p)$ and $(I^c,d+1-j)\in C(q)$.
More precisely, $\mathpzc{y}$ is not Hausdorff if and only if there are two distinct shapes $[p],[q]\in \mathpzc{y}$ which after simultaneous reordering of rows have the form 
\begin{align}
[p] & = \left[ \begin{pmatrix}
	P_{11} & P_{12} & \dots & P_{1m} \\
	0 & \vdots & \ddots & \vdots \\
	\vdots & P_{l-1,2} & \dots & P_{l-1,m} \\
	0 & \dots & 0 & P_{lm}
\end{pmatrix} \right] 
\label{T2p}
\intertext{and}
[q] & = \left[ \begin{pmatrix}
	D_1P_{11}B_1 & 0 & \dots & 0 \\
	Q_{21} & \dots & Q_{2,m-1} & \vdots \\
	\vdots & \ddots & \vdots & 0 \\
	Q_{l1} & \dots & Q_{l,m-1} & D_lP_{lm}B_m
\end{pmatrix} \right],
\label{T2q}
\end{align}
where $P_{rs}, Q_{rs}$ are matrices of the same dimensions, and 
\begin{enumerate} [label=(\roman*)]
	\item $l,m>1$ since $[p]\neq[q],$
	\item if $P_{rs},Q_{rs}\neq 0$, then $Q_{rs} = D_rP_{rs}B_s$ with $D_r$ diagonal and non-singular, $B_s$ non-singular, 
	\item $P_{rs}=0$ if there is a pair $(a,b)\neq(r,s)$ with $a\leq r,$ $b\geq s$ and $Q_{ab}\neq 0$,
	\item $Q_{rs}=0$ if there is a pair $(a,b)\neq(r,s)$ with $a\geq r,$ $b\leq s$ and $P_{ab}\neq 0$.
\end{enumerate}
\end{proposition}

Note that columns can be reordered by the right-action of $\GL$.
The form of the matrices $P$ and $Q$ in Equations~\eqref{T2p},\eqref{T2q} is illustrated in Figure~\ref{Fig:Hausdorff}.
The shape of $P$ fulfills the projective subspace constraint $(I,j)\in C(P)$ with $I$ comprising the rows of the matrix $P_{lm}$ and $j$ being $d+1$ minus the number of columns of $P_{lm}$.
Further, the shape of $Q$ fulfills the complementary projective subspace constraint $(I^c,d+1-j)\in C(Q)$.
\begin{figure}
\begin{center}
\begin{tikzpicture}[scale=0.5]

\draw[thick] (-0.2,6.1) to [out=260, in= 100] (-0.2,-0.1);
\draw[thick] (5.2,6.1) to [out=280, in= 80] (5.2,-0.1);
\fill[pattern=north east lines,pattern color=red!70] (1.025,6.025) -- (1.025,3.025) -- (2.025,3.025) -- (2.025,2.025) -- (4.025,2.025) -- (4.025,1.025) -- (5.025,1.025) -- (5.025,0.025) -- (5.025,6.025) -- (1.025,6.025);
\fill[pattern=north west lines,pattern color=blue!70] (-0.025,4.975) -- (0.975,4.975) -- (0.975,1.975) -- (2.975,1.975) -- (2.975,0.975) -- (3.975,0.975) -- (3.975,-0.025) -- (4.975,-0.025) -- (-0.025,-0.025) -- (-0.025,4.975);
\fill[green!30] (-0.025,6.025) rectangle (0.975,5.025);
\fill[green!30] (1.025,2.975) rectangle (1.975,2.025);
\fill[green!30] (3.025,1.975) rectangle (3.975,1.025);
\fill[green!30] (4.025,.975) rectangle (4.975,.025);
\end{tikzpicture}
\end{center}
\caption{The form of the matrices in Equations~\eqref{T2p} and~\eqref{T2q} of Proposition \ref{Hausdorff}.
$P$ is zero in the blue, hatched area (\protect\tikz \protect\fill[pattern=north west lines,pattern color=blue!70] (0,0) rectangle (.3,.2);) due to (iii),
 $Q$ is zero in the red, hatched area (\protect\tikz \protect\fill[pattern=north east lines,pattern color=red!70] (0,0) rectangle (.3,.2);) due to (iv).
In the green area (\protect\tikz \protect\fill[green!30] (0,0) rectangle (.3,.2);), the corresponding matrices are equivalent due to (ii).}
\label{Fig:Hausdorff} 
\end{figure}

\begin{proof}
The strategy of the proof is as follows: first, we will show that a topological non-Hausdorff subspace contains two shapes of the described form. This will be demonstrated by using the definition of Hausdorff spaces via sequences in first-countable spaces: if $p,q\in M$ with $M$ a first-countable topological space do not possess disjoint open neighborhoods, then there is a sequence with limit points $p$ and $q$.
In shape space, this gives us the sequences $([P_n])_{n\in\N},$ $([Q_n])_{n\in\N}$ with $D_nP_n=Q_nB_n$ for all $n\in\N$ and distinct limit points $[P],[Q]$.
We will show that w.l.o.g.\ $B_n$ is diagonal for all $n\in\N,$ and that the sequences $(B_n)_{n\in\N},(D_n)_{n\in\N}$ converge to singular matrices.
Different speeds of convergence lead to the described form of the limit points.

For the other direction, we will again use the idea of different speeds of convergence to construct, like in the proof of Proposition~\ref{blursplit}, a shape in any neighborhood of some $[p],[q]\in\mathpzc{y}$ of the described form. 

Now, let $[p],[q]\in \mathpzc{y}$ with $[p]\neq[q]$ such that there are no disjoint open neighborhoods of $[p]$ and $[q]$.
Since the topology of $\a_d^k$ is determined by sequences, there is a sequence $([r_n])_{n\in \N}$ in $\mathpzc{y}$ with limits $[p],[q]$.
W.l.o.g.\ $[r_n]\in\g_d^k$ for all $n\in\N$ since $\g_d^k$ is dense in $\a_d^k$ and contained in $\mathpzc{y}$. 
Thus, there are sequences $(P_n)_{n\in \N}$ with limit $P$ and $(Q_n)_{n\in \N}$ with limit $Q$ in the configuration space $\A_d^k$ such that $\pi(P_n)=\pi(Q_n)=[r_n]$ for all $n\in\N$ and $\pi(P)=[p],$ $\pi(Q)=[q]$.
Since $P_n$ and $Q_n$ have the same shape, there are non-singular diagonal matrices $D_n$ and matrices $B_n\in\GL$ such that 

$$D_nP_n=Q_nB_n $$ for all $n\in\N$.
Without loss of generality:
\begin{itemize}
	\item \textit{$B_n$ is diagonal for all $n\in\N$}: in fact, using a singular value decomposition for $B_n$, one obtains the existence of diagonal matrices $D_n, E_n$ and orthogonal matrices $U_n,V_n\in\O(d+1)$ such that $D_nP_n = Q_nV_nE_nU_n^t$ or equivalently $D_nP_nU_n = Q_nV_nE_n$.
	The sequences $(U_n)_{n\in\N}$ and $(V_n)_{n\in\N}$ have common converging subsequences since $\O(d+1)$ is compact, so w.l.o.g.\ $U_n\rightarrow U,$ $V_n\rightarrow V,$ $P_nU_n\rightarrow PU$ and $Q_nV_n\rightarrow QV$.
	Since right-multiplication by an orthogonal matrix does not change the projective shape of $P_n$ resp.\ $Q_n,$ we can choose $P_n,Q_n$ such that the corresponding $B_n$ is diagonal.
	\item \textit{$\|B_n\|_\infty =1$ for all $n\in\N$}; otherwise, consider the matrices $\|B_n\|_\infty^{-1} D_n$ and $\|B_n\|_\infty^{-1}B_n$ instead of $D_n$ and $B_n$.
	\item \textit{$(B_n)_{n\in\N}$ converges to some limit $B$ with $\|B\|_\infty=1$} since $(B_n)_{n\in\N}$ is w.l.o.g.\ bounded in the infinity norm, hence possesses at least a converging subsequence. Thus $Q_nB_n\rightarrow QB$.
	\item \textit{$(D_n)_{n\in\N}$ converges to some limit $D$}, hence $\|D_n\|_\infty\leq\rho,$ $\rho>0,$ for all $n\in\N$; else, since $D_nP_n\rightarrow QB$ and $P_n\rightarrow P$, a row of $P$ would be the null vector which is impossible.
	\item \textit{$B$ and $D$ are singular, but non-trivial, i.e., $B,D\neq 0$}: if $B$ is non-singular, so is $D$ since, otherwise, $QB$ and thus $Q$ would have a vanishing row which is impossible. If $D$ is non-singular, so is $B$ since, otherwise, $P$ would be of rank less than $d+1$ in contradiction to the assumption $\mathpzc{y}\subseteq\frk_d^k$.
	If both are non-singular, then $P=D^{-1}QB$ in contradiction to $[p]\neq[q]$.
	$B$ is non-trivial since $\|B\|_\infty =1$, while $D$ is non-trivial since $B$ is non-trivial and $P$ and $Q$ are of full rank.
\end{itemize}

By reordering of rows and columns and considering subsequences if necessary, one may assume that $\Bigl(\tfrac{(D_n)_{ii}}{(D_n)_{jj}}\Bigr)_{n\in\N}$ and $\Bigl(\tfrac{(B_n)_{ii}}{(B_n)_{jj}}\Bigr)_{n\in\N}$ converge to a finite limit for all $i<j$.
By merging columns respectively rows of equal speed of convergence to 0 (resp.\ columns/rows converging to non-zero values) into a block labeled $(r,s)$, one derives the proposed block structure of $P$ and $Q$. 
Blocks of type (ii) may arise if $\Bigl(\tfrac{(D_n)_{ii}}{(B_n)_{jj}}\Bigr)_{n\in\N}$ converges to a non-zero value for some, and hence all $(i,j)$ in block $(r,s)$.
If the sequence $\Bigl(\tfrac{(D_n)_{ii}}{(B_n)_{jj}}\Bigr)_{n\in\N}$ converges to 0, then $Q_{ij}=0$ which explains type (iv). 
For the blocks of type (iii), consider the equalities $P_nF_n=G_nQ_n$ with $F_n=\|B_n^{-1}\|_\infty^{-1} B_n^{-1}$ and $G_n=\|B_n^{-1}\|_\infty^{-1} D_n^{-1}$ for all $n\in\N$.
If the sequence $\Bigl(\tfrac{(D_n)_{ii}}{(B_n)_{jj}}\Bigr)_{n\in\N}$ diverges, or equivalently, the sequence $\Bigl(\tfrac{(B_n)_{jj}}{(D_n)_{ii}}\Bigr)_{n\in\N} = \Bigl(\tfrac{\|B_n\|^{-1}_\infty(D^{-1}_n)_{ii}}{\|B_n\|^{-1}_\infty(B^{-1}_n)_{jj}}\Bigr)_{n\in\N} = \Bigl(\tfrac{(G_n)_{ii}}{(F_n)_{jj}}\Bigr)_{n\in\N}$ converges to 0, then $P_{ij}=0$ which explains type (iii).
Recall that neither $P$ nor $Q$ may have trivial rows or columns by assumption.

Finally, we have to show that the upper left and bottom right blocks are of type (ii): since every row of $Q$ is non-trivial, $\Bigl(\tfrac{(D_n)_{kk}}{(B_n)_{d+1,d+1}}\Bigr)_{n\in\N}$ does not converge to 0. Since $P$ is of full rank, the sequence of inverses $\Bigl(\tfrac{(B_n)_{d+1,d+1}}{(D_n)_{kk}}\Bigr)_{n\in\N}$ does not converge to 0, whence it converges to a non-zero number. Analogously, $\Bigl(\tfrac{(B_n)_{11}}{(D_n)_{11}}\Bigr)_{n\in\N}$ converges to a non-zero number since $P$ has no row of zeroes, and $Q$ is of full rank. This finishes the proof that $[p],[q]$ are of the described form.

Conversely, assume there exist $[P],[Q]\in \mathpzc{y}$ with $P,Q$ in the described form, and let $U_{[p]}$ and $U_{[q]}$ be open neighborhoods of $[p]$ resp.\ $[q]$.
Then there is a $\delta>0$ such that $B_\delta(P)\subseteq \pi^{-1}(U_{[p]})$ and $B_\delta(Q)\subseteq \pi^{-1}(U_{[q]})$ in the space of (matrix) configurations.
We will construct a configuration $A$ which is an element of both  $B_\delta(P)$ and $B_\delta(Q)$.
For $n\in\N,$ consider block diagonal matrices 
\begin{equation*}
 \tilde{D}_n = \begin{pmatrix}
	n^{d_1} \tilde{D}_1 & \dots & 0 \\
	\vdots & \ddots & \vdots \\
	0 & \dots & n^{d_l} \tilde{D}_l
\end{pmatrix}
\end{equation*}
and
\begin{equation*}
\tilde{B}_n = \begin{pmatrix}
	n^{-b_1} \tilde{B}_1 & \dots & 0 \\
	\vdots & \ddots & \vdots \\
	0 & \dots & n^{-b_m} \tilde{B}_m
\end{pmatrix}
\end{equation*}
 with non-singular diagonal matrices $\tilde{D}_r,$ non-singular matrices $B_s$ and speeds of convergence $d_r,b_s\in\N_0$ such that
\begin{itemize}
	\item $b_r>b_s,$ $d_r>d_s$ for all $r>s$;
	\item $d_r=b_s$ and $\tilde{D}_r=D_r, \tilde{B}_s=B_s$ for pairs $(r,s)$ with $P_{rs},Q_{rs}\neq 0,$ and thus $Q_{rs}= D_rP_{rs}B_s$;
	\item $b_s\neq d_r$ and $\tilde{D}_r= \mbox{Id}, \tilde{B}_s= \mbox{Id}$ else; more precisely, let $d_r< b_s$ for all $(r,s)$ with $P_{rs}\neq0,$ while $d_r > b_s$ for all $(r,s)$ with $Q_{rs}\neq0$.
\end{itemize} Next, define the matrix $A=(A_{rs})$ with the same block structure as $P,Q$ and entries 
\begin{equation*}
 A_{rs}  = \begin{cases} 
P_{rs} & \mbox{if } P_{rs}\neq0, \\ 
n^{b_s-d_r} \tilde{D}_r^{-1} Q_{rs} \tilde{B}_s^{-1} & \mbox{if } P_{rs}= 0. \end{cases}
\end{equation*}
Then,
\begin{equation*}
 \bigl( \tilde{D} A \tilde{B} \bigr)_{rs}  = \begin{cases} 
Q_{rs} & \mbox{if } P_{rs}=0, \\ 
n^{d_r-b_s} \tilde{D}_r P_{rs} \tilde{B}_s & \mbox{if } P_{rs}\neq 0. \end{cases} 
\end{equation*}
Moreover, 
\begin{equation*} 
\max \bigl\{ n^{b_s-d_r} : (r,s) \mbox{ with } Q_{rs}\neq 0, P_{rs}=0 \bigr\} \leq n^{-1}
\end{equation*}
and
\begin{equation*}
\max \bigl\{ n^{d_r-b_s} : (r,s) \mbox{ with } P_{rs}\neq 0, Q_{rs} = 0 \bigr\} \leq n^{-1}.
\end{equation*}
Now, choose $n$ large enough such that 
\begin{equation*}
n^{-1}\cdot\max_{(r,s)}\bigl\{ \bigl\|\tilde{D}_r P_{rs} \tilde{B}_s\bigr\|_\infty, \bigl\|\tilde{D}_r^{-1} Q_{rs} \tilde{B}_s^{-1}\bigr\|_\infty \bigr\} <\delta,
\end{equation*}
 whence $A \in B_\delta(P)\cap B_\delta(Q)$, i.e.\  $B_\delta(P)\cap B_\delta(Q)\neq\emptyset$ as subsets of $\A_d^k$.
Since $\G_d^k$ is dense in $\A_d^k,$ there is an $\tilde{A}\in \G_d^k$ with $\tilde{A}\in B_\delta(P)\cap B_\delta(Q)$ whence $[\tilde{A}]\in U_{[p]}\cap U_{[q]}$.
Therefore, $\mathpzc{y}$ is not Hausdorff.
\end{proof}

Proposition~\ref{Hausdorff} shows that $\fr_d^k$ is not Hausdorff: the configurations 
\begin{equation*}
P=\begin{pmatrix}
	1 & 1 & \cdots & 1 \\
	1 & 0 & \cdots & 0 \\
	0 & 1 & \ddots & \vdots \\
	\vdots & \ddots & \ddots & 0 \\
	0 & \cdots & 0 & 1 \\
	0 & \cdots & 0 & 1 \\
	\end{pmatrix}\quad
\mbox{and}\quad Q=\begin{pmatrix}
	1 & 0 & \cdots & 0 \\
	1 & 0 & \cdots & 0 \\
	0 & 1 & \ddots & \vdots \\
	\vdots & \ddots & \ddots & 0 \\
	0 & \cdots & 0 & 1 \\
	1 & \cdots & 1 & 1 \\
\end{pmatrix}
\end{equation*}
are in $\Fr_d^{d+3}$.
Thus, $\fr_d^{d+3}$ is not Hausdorff since $[P]$ and $[Q]$ are of the described form of Proposition~\ref{Hausdorff}.
For $k>d+3,$ some of the landmarks may be repeated.
In particular, $\f_d^k\supseteq\fr_d^k$ is not Hausdorff.

\begin{example}\label{ex:Hausdorff}
In the case $d=1$ and $k=4,$ Proposition~\ref{Hausdorff} states that e.g.\ the topological subspace $\g_1^4$ (no coincidences) together with the single pair coincidences $[p_{3=4}]$ with three distinct landmarks $p_1,p_2,p_3$ but $p_3=p_4$ and $[q_{1=2}]$ with three distinct landmarks $q_2,q_3,q_4$ but $q_1=q_2$, though being T1, is not Hausdorff.
In fact, then
\begin{equation*}
[p_{3=4}]=\left[ \begin{pmatrix}
	1 & 1 \\ 1 & 0 \\ 0 & 1 \\ 0 & 1 
\end{pmatrix} \right]
\quad\mbox{and}\quad
[q_{1=2}]=\left[ \begin{pmatrix}
	1 & 0 \\ 1 & 0 \\ 0 & 1 \\ 1 & 1 
\end{pmatrix} \right]
\end{equation*}
as before.
Thus, $\g_1^4=\t_1^4$ is the only maximal Hausdorff subspace closed under relabeling and respecting the hierarchy of subspace constraints since any shape with a non-trivial subspace constraint features a point coincidence.
\end{example}

\section{Topological subspaces bounded by projective subspace numbers}\label{SubNum}

Proposition~\ref{Hausdorff} shows again that a space of shapes with a fixed pseudo-frame is a Hausdorff manifold.
However, these kind of spaces are not closed under relabeling, i.e., they do not fulfill requirement (b) of the introduction.
As a remedy we introduce the idea of bounding the number of landmarks in a projective subspace depending on its dimension.

To a vector $n=(n_1,\dots,n_d)\in\N^d$ with $1 \leq n_1 < n_2 < \dots < n_d$ define the topological sub\-space 
\begin{equation}
\SN_{\;d}^k(n) = \big\{p\in\A_d^k:|I| \leq n_j \mbox{ for all } (I,j)\in C(p) \big\} ,
\end{equation}
i.e., $\SN_{\;d}^k(n)\subseteq\A_d^k$ comprises those configurations $p$ for which there will be at most $n_j$ landmarks in any $(j-1)$-dimensional projective subspace of $\RP^d$.
We will then say the topological subspace $\SN_{\;d}^k(n)$ is \emph{bounded by the projective subspace numbers} $n$.
Note that $\SN_{\;d}^k(n)$ is closed under permutations and respects the hierarchy of projective subspace constraints, i.e.\ requirement (c) in the introduction, and contains $\G_d^k$ since $n_j\geq j$ for all $1\leq i\leq d,$ while $\SN_{\;d}^k(n)=\G_d^k$ if and only if $n_j= j$ for all $1\leq j\leq d,$ and $\SN_{\;d}^k(n)=\A_d^k$ if and only if $n_j\geq k$ for all $1\leq j\leq d$.

We are interested in projective subspace numbers $n$ which lead to Hausdorff spaces $\sn_d^k(n)$.
From Proposition~\ref{Hausdorff}, we can infer conditions for feasible $n\in\N^d$ under which the corresponding shape space $\sn_d^k(n)$ is a Hausdorff manifold.

\begin{theorem}\label{SubspaceNumbers}
Consider projective subspace numbers $n=(n_1,\dots,n_d)$. The following statements are equivalent:
\begin{enumerate}[label=(\roman*)]
	\item $\sn_d^k(n)$ is Hausdorff;
	\item $\sn_d^k(n)\subseteq\f_d^k$;
	\item $\sn_d^k(n)$ is an open, Hausdorff submanifold of $\f_d^k$;
	\item $n_j + n_{d+1-j} < k$ for all $1\leq j\leq d$.
\end{enumerate}
\end{theorem}

\begin{proof}
First, assume (iv) $n_j + n_{d+1-j} < k$ for all $1\leq j\leq d$.
If $\sn_d^k(n)$ were not Hausdorff, there would be shapes $[p],[q]\in\sn_d^k(n)$ as in Proposition~\ref{Hausdorff} with $(I,j)\in C(p)$ and $(I^c,d+1-j)\in C(q)$ for some $I\subseteq\{1,\dots,k\}$ and some $j\in\{1,\dots,d\}$ with $|I|\leq n_j$ and $|I^c|\leq n_{d+1-j}$.
But then $k=|I|+|I^c|\leq  n_j + n_{d+1-j}$ in contradiction to the assumption.
Additionally, $\sn_d^k(n)\subseteq\f_d^k$ since $\sn_d^k$ does not contain any splittable shapes.
Further, $\sn_d^k$ is an open subset of $\f_d^k$ since it respects the hierarchy of projective subspace constraints, see Section~\ref{free}.
Thus, $\sn_d^k$ is a submanifold of the differentiable manifold $\f_d^k$ (Theorem~\ref{MnfOfFree}) and (iii) holds.

Conversely, assume that $n_j + n_{d+1-j} \geq k$ for some $1\leq j\leq d$.
Then, there are shapes $[p]\in\sn_d^k(n)$ with $(I,j),(I^c,d+1-j)\in C(p)$. But those shapes are splittable and $\{[p]\}\neq\Bl([p])\subseteq\sn_d^k(n)$ whence $\sn_d^k(n)$ is not even T1.
\end{proof}

Now, there is a canonical partial order on $\N^d$ induced by the component-wise total order on $\N$.
We call a vector $n\in\N^d$ \emph{maximal} if $\sn_d^k(n)$ is Hausdorff and $\sn_d^k(m)$ is not Hausdorff for any $m>n$ with respect to that partial order.
This notion of maximality accords with requirement (d) of the introduction.

Note that $\g_d^k$ is bounded by projective subspace numbers $n_j=j$ for $j\in\{1,\dots,d\}$ whence $g_d^k$ is a Hausdorff manifold since $k\geq d+3$.
However, this topological subspace is not maximal unless $d=1$ and $k=4$, since then $n_1 + n_{d} = d+1,$ so $n_d$ can be increased by 1 without violating Theorem~\ref{SubspaceNumbers}(iv) if $d>1$, or $n_1$ and $n_d$ if $k>d+3$.

\section{Tyler regular shapes}\label{Geometry}

The space $\t_d^k$ of Tyler regular shapes (cf.\ Section~\ref{approaches}) is a differentiable Hausdorff manifold since $\t_d^k = \sn_d^k(t)$ is bounded by the projective subspace numbers $t=(t_1,\dots,t_d)$ with 
\begin{equation}\label{TylerNumbers}
t_j = \left\lceil \frac{jk}{d+1} \right\rceil -1 \quad\mbox{for all } j\in\{1,\dots,d\},
\end{equation}
and thus for these values $t_j + t_{d+1-j} < \frac{jk}{d+1} + \frac{(d+1-j)k}{d+1} = k$.
In fact, $\t_d^k$ is maximal for some choices for $k$ and $d$.

\begin{proposition}
The vector $t\in\N^d$ in Equation~\eqref{TylerNumbers} of projective subspace numbers of $\t_d^k$ is maximal if and only if the greatest common divisor of $k$ and $d+1$ is either 1 or 2. In particular, $\t_d^k$ is maximal for
\begin{enumerate}[label=(\roman*)]
	\item $d=1$ and arbitrary $k\geq d+3,$
	\item arbitrary $d$ and $k=d+3$, as well as
	\item relatively prime $k$ and $d+1$.
\end{enumerate}
\end{proposition}

\begin{proof}
If $k$ and $d+1$ are relatively prime, then $t_j + t_{d+1-j}=k-1$ for all $1\leq j\leq d$ due to rounding. More precisely, $t_j + t_{d+1-j}=k-1$ if $\tfrac{jk}{d+1}$ is not integral.
Otherwise $k$ and $d+1$ have a greatest common divisor $c>1,$ so $j<d+1$ needs to be a multiple of $\tfrac{d+1}{c}$.
However,
\begin{align*}
t_{(d+1)/c}+t_{d+1-(d+1)/c} & = \left\lceil \frac{d+1}{c}\frac{k}{d+1} \right\rceil -1  + \left\lceil (c-1)\frac{d+1}{c}\frac{k}{d+1} \right\rceil -1 \\
& = \frac{k}{c} + (c-1)\frac{k}{c} -2 \\
& = k-2,
\end{align*}
whence $t_{d+1-(d+1)/c}$ (and its successors if necessary) can be increased by 1 for $c\neq 2$ without violating Theorem~\ref{SubspaceNumbers}(iv).
In case $c=2$, though, $ j= \tfrac{d+1}{2} = d+1-\tfrac{d+1}{2} $ is the only projective subspace dimension for which $\tfrac{jk}{d+1}$ is integral and $t_{(d+1)/2}= t_{d+1-(d+1)/2} = \tfrac{k-2}{2}$ cannot be increased.
\end{proof}

For example, $\t_2^6$ is not maximal. Here, $t=(1,3)$ which is not maximal since both $n=(1,4)$ and $m=(2,3)$ are larger and do not violate Theorem~\ref{SubspaceNumbers}(iv).

Since $\t_d^k$ is a differentiable Hausdorff manifold, it may be equipped with a Rie\-man\-nian metric, for example in the following way.

Recall that any Tyler regular configuration $p\in\T_d^k$ has a matrix representation $P$ fulfilling 
\begin{equation} 
P^t P = \Id_{d+1} \label{TylerColumn}
\end{equation}
and 
\begin{equation} 
P_{i\cdot} P_{i\cdot}^t=\tfrac{d+1}{k} \quad \mbox{ for all } i\in\left\{1,\dots,k\right\} \label{TylerRow}
\end{equation}
(see Section~\ref{approaches}).
Again, this Tyler standardization $P$ is only unique up to multiplication of the rows by~$\pm 1$ and right-multiplication by an orthogonal matrix, i.e.\ unique up to a compact group action, and can be considered as a projective pre-shape.
Even more, we can remove the action of the orthogonal group by passing to the $k\times k$-dimensional matrix $PP^t$.

Now, the space of Tyler standardized configurations of Tyler regular shapes is a submanifold of $\R^{k\times (d+1)}$ and therefore naturally inherits a Riemannian metric from $\R^{k\times (d+1)}$.
Since every element of the remaining group action acts as an isometry on $\R^{k\times (d+1)}$, the push-forward of the Riemannian metric on the space of Tyler standardized configurations to $\t_d^k$ is a Riemannian metric on $\t_d^k$.

This standardization suggests itself through the following geometric reasoning: consider a shape $[P]\in\frk_d^k$ of full rank and one of its matrix configurations $P$.
By definition, $P$ is only unique up to left-multiplication with non-singular, diagonal $k\times k$-dimensional matrices and right-multiplication of non-singular $(d+1)\times (d+1)$-dimensional matrices.
Indeed, we can view the columns of the $k\times (d+1)$-dimensional matrix $P$ as a basis of a $(d+1)$-dimensional linear subspace of $\R^k,$ and the action of $\GL$ as a change of basis.
In particular, we can choose an orthonormal basis of the column space as a representation, i.e.\ a matrix $P$ with orthonormal columns.
Then, $P^tP=\Id_{d+1}$ with $P$ being unique up to the action of $\O(d+1)$ from the right.

Following this line of thought, we can think of the left-action of diagonal matrices as an action on the Grassmannian manifold $\Gr(k,d+1)$ of $(d+1)$-di\-men\-sio\-nal linear subspaces of $\R^k$. 
Of course, elements of the Grassmannian $\Gr(k,d+1)$ can be represented by the corresponding projection matrices
\begin{equation}
\mathcal{P}_P = P \bigl(P^tP\bigr)^{-1} P^t
\end{equation}
onto the column space of $P$.
This is the so-called Vero\-ne\-se-Whitney embedding of $\Gr(k,d+1)$ into $\R^{k\times k}$.
$\mathcal{P}_P$ is then a $k\times k$-dimensional matrix of rank and trace $d+1$.
In this representation, the action of diagonal matrices on the Grassmannian acts infinitesimally like certain rotations in $\R^k$:
for a non-singular diagonal matrix $D=\mbox{diag}(D_i)_{i=1,\dots,k}$ in a sufficiently small neighborhood of $\Id_k$ use
\begin{align*}
\bigl(P^tD^2P\bigr)^{-1} & = \bigl(\Id_{d+1}-(\Id_{d+1}-P^tD^2P)\bigr)^{-1} \\ & = \sum_{n=0}^\infty{\bigl(\Id_{d+1}-P^tD^2P\bigr)^n}
\end{align*}
and $\tfrac{\partial}{\partial D_i} D=e_ie_i^t$ with $e_i$ being the $i$-th canonical basis vector of $\R^k$ to obtain
\begin{align*}
\tfrac{\partial}{\partial D_i}&\mathcal{P}_{DP}  ={} \tfrac{\partial}{\partial D_i} DP\bigl(P^tD^2P\bigr)^{-1} P^tD \\
={}& e_ie_i^t P\bigl(P^tD^2P\bigr)^{-1} P^tD  + DP\bigl(P^tD^2P\bigr)^{-1} P^te_ie_i^t \\
& + DP \Biggl[ \sum_{n=1}^\infty{ \sum_{l=1}^n{ \bigl(\Id_{d+1}- P^tD^2P\bigr)^{l-1} \bigl(-2D_iP^te_ie_i^tP\bigr) \bigl(\Id_{d+1}- P^tD^2P\bigr)^{n-l} }} \Biggr]P^tD.
\end{align*}
For $D=\Id_k,$ $P^tP=\Id_{d+1},$ and consequently $P^tD^2P=\Id_{d+1},$ $D_i=1,$ $\mathcal{P}_{P}^2=\mathcal{P}_{P}=PP^t$, we conclude
\begin{align*}
\tfrac{\partial}{\partial D_i}\mathcal{P}_{DP} & = e_ie_i^t PP^t + PP^t e_ie_i^t - 2PP^te_ie_i^tPP^t \\
 & = \bigl(\underbrace{e_ie_i^t \mathcal{P}_{P} - \mathcal{P}_{P}e_ie_i^t}_{\mbox{antisymmetric}}\bigr) \mathcal{P}_{P} + \mathcal{P}_{P} \bigl( \underbrace{\mathcal{P}_{P}e_ie_i^t - e_ie_i^t\mathcal{P}_{P}}_{\mbox{antisymmetric}}\bigr),
\end{align*}
while the infinitesimal action of the orthogonal group $\O(k)$ acting by conjugation is given by 
\begin{align*}
\tfrac{\partial}{\partial t}|_{t=0}O(t)\mathcal{P}_{P}O(t)^t & = \dot{O}(0)\mathcal{P}_{P} + \mathcal{P}_{P}\dot{O}^t(0) \\
& = \dot{O}(0)\mathcal{P}_{P} - \mathcal{P}_{P}\dot{O}(0)
\end{align*}
for a differentiable curve $\R\ni t \mapsto O(t)\in\O(k),$ $O(0)=\Id_k$ with antisymmetric $\dot{O}(0)\in \mathfrak{o}(k)=\{M\in\R^{k\times k}: M=-M^t\}$.
Hence, the diagonal matrices act infinitesimally like certain rotations. 

In fact, $\tfrac{\partial}{\partial D_i}\mathcal{P}_{DP}$ is an infinitesimal rotation in the plane spanned by $\mathcal{P}_P e_i$ and $e_i$.
This suggests to fix the angle 
\begin{equation*}
\bigl\langle e_i, \mathcal{P}_Pe_i\bigr\rangle = e_i^t\mathcal{P}_Pe_i = e_i^tPP^te_i = P_{i\cdot}P_{i\cdot}^t
\end{equation*}
for all $1\leq i\leq k$ in order to standardize the projection matrix $\mathcal{P}_P$ and thus the configuration $P$.
Of course, we require invariance under permutations whence all directions $e_i$ resp.\ landmarks $P_{i\cdot}$ have to be treated equally, i.e.,
\begin{equation}
P_{i\cdot}P_{i\cdot}^t = C\in\R
\end{equation} 
for all $1\leq i\leq k$.
The constant $C$ has to be $\tfrac{d+1}{k}$ since the values $P_{i\cdot}P_{i\cdot}^t$ are the diagonal elements of $\mathcal{P}_P$ and $\mathcal{P}_P$ has trace $d+1$ as it is the orthogonal projection onto a $(d+1)$-dimensional linear subspace. We thus obtain Equation~\eqref{TylerRow}.

This discussion of Tyler standardization shows that the topological subspace of Tyler regular shapes is a topological subspace of the quotient of a Grassmannian with a finite group action (multiplication of the rows by $\pm1$) whence we can obtain a Riemannian metric on this space by considering one on the Grassmannian: the tangent space at the point $\mathcal{P}_P = P \bigl(P^tP\bigr)^{-1} P^t$ is 
\begin{equation}
\Bigl\{ \bigl[\mathcal{P}_P, A \bigr] \, : A\in\mathbf{so}(k), \mbox{diag}\bigl[\mathcal{P}_P, A \bigr] = 0 \Bigr\},
\end{equation}
with the standard Riemannian metric $\langle A,B\rangle = \mbox{tr}(A^tB)$ on $\R^{k\times k}$ which up to a constant induces the very metric given above.
 
A result by Tyler \cite{Tyl}, cf.\ \cite{KM}, shows that Tyler standardization is possible for the Tyler regular shapes defined in Section~\ref{approaches};
the only other ones for which it is possible are those splittable shapes $[p]$ for which $|I|=\tfrac{jk}{d+1}$ and $|I^c|=\tfrac{(d+1-j)k}{d+1}$ for any projective subspace constraint $(I,j)\in C(p)$ with $(I^c,d+1-j)\in C(p)$ and $|I|<\tfrac{jk}{d+1}$ otherwise. 
The latter can obviously only exist when $d+1$ and $k$ have a common divisor.
The space of projective shapes which allow Tyler standardization then does not respect the hierarchy of projective subspace constraints if there exists such a splittable Tyler standardizable shape.
However, it can be shown to be closed under permutations and a differentiable manifold by identifying these splittable configurations with those in its blur.
Unfortunately, it is unclear if the Riemannian metric given above can be extended to this subspace since the remaining discrete group action is not free on the splittable Tyler standardized configurations.
Even worse, the metric on $\t_d^k$ given above is not complete if splittable Tyler standardizable shapes exist.

If $\t_d^k$ is not maximal, i.e., if and only if $d+1$ and $k$ have a common divisor greater than~2, then $\t_d^k$ is a submanifold of a larger feasible topological subspace bounded by projective subspace numbers.
However, the Riemannian metric on $\t_d^k$ given above cannot be extended to the larger topological subspace since the elements lying in the blur of a splittable Tyler standardizable shape would have distance 0 in this extension, i.e., the extension cannot be a metric.

\begin{example}
In the case $d=1$ and $k=4$, the Tyler standardizable shapes are the Tyler regular ones, i.e.\ those in general position, and the three splittable shapes with double pair coincidences (type (c) from Section~\ref{case14})
\begin{equation*}
\left[\begin{pmatrix} 1 & 0 \\ 1 & 0 \\ 0 & 1 \\ 0 & 1 \end{pmatrix}\right],\quad \left[\begin{pmatrix} 1 & 0 \\ 0 & 1 \\ 1 & 0 \\ 0 & 1 \end{pmatrix}\right] \quad\mbox{and}\quad \left[\begin{pmatrix} 1 & 0 \\ 0 & 1 \\ 0 & 1 \\ 1 & 0 \end{pmatrix}\right].
\end{equation*}
\end{example}

\section{Discussion}

The subject of this article was to find a reasonable differentiable Hausdorff submanifold of projective shape space.
It turns out that the topological subspace comprising shapes of configurations with trivial isotropy group is only a differentiable T1 manifold, but not Hausdorff in contrast to the situation in similarity and affine shape spaces, cf.\ \cite{DM,KBCL} resp.\ \cite{GT,PM}.
Charts were constructed by introducing the concept of pseudo-frames generalizing the well-known notion of projective frames.

Additionally, by bounding the number of landmarks per projective subspace of $\RP^d$, a new class of reasonable topological subspaces, namely those bounded by projective subspace numbers, was introduced.
For this class, a criterion was given for deciding whether these topological subspaces are differentiable Hausdorff manifolds.
Indeed, one of these topological subspaces has been considered in literature before, namely the space of Tyler regular shapes.
By Tyler standardization, for which we presented new, geometric arguments, this topological subspace can be endowed with a Riemannian metric.
When it is maximal in the class of topological subspaces bounded by projective subspace numbers, one could say that it fulfills all of the requirements except that the Riemannian metric might not be complete.

However, it remains unclear how to endow other topological subspaces with a complete Riemannian metric, in particular in cases where the topological subspace of Tyler regular shapes is not maximal.

\paragraph{Funding.} F. Kelma gratefully acknowledges financial support by Klaus Tschira Stiftung gGmbH, project 03.126.2016. The funding source had no direct involvement in the conduct of this research.
\paragraph{Declarations of interest.} None.

\bibliographystyle{elsarticle-num}
\bibliography{proj}

\end{document}